\documentclass[10pt]{amsart}

\usepackage{amsmath,amssymb}
\usepackage[hidelinks]{hyperref}

\numberwithin{equation}{section}

\newtheorem{theorem}{Theorem}[section]
\newtheorem{proposition}[theorem]{Proposition}

\newtheorem{corollary}[theorem]{Corollary}

\title[Alexander quandle rings]{Group-ring methods for Alexander quandle rings}

\author{Zhi-Lin Zhang}
\address{Independent Researcher, Taipei, Taiwan}
\email{hsa00000@gmail.com}


\begin{document}

\begin{abstract}
For an Alexander quandle on an abelian group $G$ with automorphism
$\varphi$, we identify its augmentation powers and nonunital commutator
subalgebra with explicit ideals in the ordinary group ring over any
commutative unital coefficient ring.  Each left-normed augmentation power
is also the span of all products with the same number of factors,
independently of parenthesization.  Over $\mathbb{Z}$, we determine every
successive augmentation quotient for Alexander quandles on nontrivial
finite cyclic groups.  For an even dihedral quandle of order $m\geq4$,
every quotient in degree $r\geq2$ is
$(\mathbb{Z}/(m/2)\mathbb{Z})^2$; this proves the proposed replacement for
the disproved order-$m$ prediction.  For the integral quandle ring of
$\operatorname{Core}(\mathbb{Z})$, we describe all augmentation powers,
determine the nonunital commutator subalgebra, and prove that the difference
of the basis elements indexed by $1$ and $0$ does not lie in the square of
the augmentation ideal.  When $(1-\varphi)G$ is a finite $p$-group, the
integral filtration is separated, so every nonzero integral idempotent has
augmentation one.  Fourier analysis shows that finite odd-order
commutative Alexander quandles have only zero and the basis elements as
integral idempotents.  Combined with the midpoint description of medial
commutative quandles, this gives the same classification for every finite
medial commutative quandle.  Finally, the complex quandle algebra of the
core of any nontrivial finite abelian group of odd order has right Peirce
spectrum $\{0,1,-1\}$, proving the conjectured odd-order dihedral spectrum.
\end{abstract}

\maketitle

\section{Introduction}

Although quandle rings need not be associative, multiplication in an
Alexander quandle ring factors through endomorphisms of an ordinary
commutative group ring.  This factorization converts several filtration,
commutator, and idempotent problems into ordinary ideal theory and Fourier
analysis.

We use right-distributive quandles: each right translation
$R_y(x)=x*y$ is an automorphism fixing $y$; see~\cite[p.~89]{EN15}.
For a quandle $Q$ and a commutative ring $k$ with identity, the quandle ring
$k[Q]$ is the free $k$-module with basis $\{e_x:x\in Q\}$ and bilinear
multiplication $e_xe_y=e_{x*y}$.  The coefficient-sum homomorphism
$\varepsilon_Q(\sum_x c_xe_x)=\sum_x c_x$ has kernel $\Delta$, the
augmentation ideal.  Quandle rings were introduced in~\cite{BPS19};
augmentation powers, idempotents, automorphisms, commutators, and Peirce
spectra have subsequently been studied
in~\cite{BE26,BPS22,CCFK26,ENS23,ST26} and related work.

Section~\ref{sec:model} factors the multiplication for an Alexander
quandle defined by an automorphism $\varphi$ of an abelian group $G$
through two endomorphisms of the ordinary group ring.  The first main
structural result, Theorem~\ref{thm:alexander-powers}, identifies every
left-normed augmentation power with a product of the ordinary augmentation
ideal and powers of the subgroup ideal determined by $(1-\varphi)G$; it
also gives the same submodule for all parenthesized products with the same
number of factors.  Theorem~\ref{thm:commutator-subalgebra} identifies the
nonunital subalgebra generated by all commutators $uv-vu$ with the subgroup
ideal determined by $(1-2\varphi)G$.  Both formulas hold over every
commutative unital coefficient ring.

Theorem~\ref{thm:cyclic-quotients} determines every successive integral
augmentation quotient for Alexander quandles on nontrivial finite cyclic
groups.  Its dihedral specialization recovers the odd-order quotients and the
even-order first quotient of~\cite[Theorems~6.2 and~6.3]{EFT19}.
For even order $m\geq4$, Corollary~\ref{cor:dihedral-quotients} gives
$(\mathbb{Z}/(m/2)\mathbb{Z})^2$ in every degree $r\geq2$, for left-normed
powers and hence also for the spans over all parenthesizations.
This proves~\cite[Conjecture~2]{PP24}, the replacement proposed after
the order-eight counterexample in~\cite[Theorem~1]{PP24} disproved
the order-$m$ prediction of~\cite[Conjecture~6.5(2)]{BPS19}.

For the infinite cyclic core quandle,
Corollary~\ref{cor:infinite-core} describes all augmentation powers and
proves $e_1-e_0\notin\Delta^2$.  Thus it describes the square asked for
in~\cite[Question~3.4]{BE26} and answers the membership subquestion
negatively.  The same corollary determines the commutator subalgebra
and shows that the quotient of $\Delta$ by this subalgebra is
$\mathbb{Z}^2$, giving a negative answer to the separate reverse-inclusion
question in~\cite[Section~3.1]{BE26}.

Theorem~\ref{thm:p-primary-separation} proves that the integral
augmentation filtration is separated whenever $(1-\varphi)G$ is a finite
$p$-group, without requiring $G$ itself to be finite.  Consequently every
nonzero integral idempotent has augmentation one.  This establishes
\cite[Conjecture~6.2]{ENS23} for this class and, by
Corollary~\ref{cor:cyclic-type}, for every finite quandle of cyclic type.

Section~\ref{sec:fourier} treats idempotents and right multiplication
operators in Fourier coordinates.  Theorem~\ref{thm:midpoint-idempotents}
classifies the integral idempotents of finite odd-order commutative
Alexander quandles.  Ta proved that every medial commutative quandle is
isomorphic to a midpoint quandle~\cite[Corollary~4.2]{Ta26}; hence
Corollary~\ref{cor:medial-commutative-idempotents} extends the
classification to every finite medial commutative quandle.  The
cyclic-group specialization answers the integral instance
of~\cite[Question~5.5]{BE26}
and proves the same classification for every odd order, beyond the
question's prime-order range.  Finally, Theorem~\ref{thm:core-peirce}
gives right Peirce spectrum $\{0,1,-1\}$ for the complex quandle algebra
of the core of every nontrivial finite abelian group of odd order.  Its
dihedral specialization, Corollary~\ref{cor:dihedral-peirce},
proves~\cite[Conjecture~7.6]{ENS23}.

\section{Group-ring transport and universal formulas}
\label{sec:model}

We first express the quandle product in the ordinary group ring, then
identify the ideals controlling augmentation products and commutators.

\subsection{The transported multiplication}

Let $k$ be a commutative ring with identity, let $G$ be an abelian group
written additively, and let $\varphi\in\operatorname{Aut}(G)$.  The
Alexander quandle $Q=\operatorname{Alex}(G,\varphi)$ has operation
$x*y=\varphi(x)+(1-\varphi)y$, where $1$ denotes the identity
endomorphism of $G$.  Write $X^g$ for the group-ring basis element indexed
by $g$, so the ordinary product satisfies $X^gX^h=X^{g+h}$.  Let
\[
 I=\ker\bigl(\varepsilon:k[G]\to k\bigr)
\]
be the augmentation ideal of the ordinary group ring.  Define a group-ring
automorphism $\alpha$ and a group-ring endomorphism $\beta$ by
\[
 \alpha(X^g)=X^{\varphi(g)},
 \qquad
 \beta(X^g)=X^{(1-\varphi)g}.
\]
For $u,v\in k[G]$, set $u\star v=\alpha(u)\beta(v)$, where juxtaposition on
the right denotes the ordinary product in $k[G]$.

\begin{proposition}[Group-ring transport]
\label{prop:transported-product}
The $k$-module isomorphism
$\Phi:k[Q]\to k[G]$, $e_g\mapsto X^g$, is an algebra isomorphism from
$k[Q]$ onto $(k[G],\star)$.
\end{proposition}

\begin{proof}
By $k$-bilinearity, it suffices to observe that for $g,h\in G$,
\[
 \Phi(e_g e_h)
 =X^{\varphi(g)+(1-\varphi)h}
 =\alpha(X^g)\beta(X^h).\qedhere
\]
\end{proof}

For a subgroup $M\leq G$, write
\[
 L_M=\left\langle X^m-1:m\in M\right\rangle_{k[G]}
     =\ker\bigl(k[G]\longrightarrow k[G/M]\bigr).
\]

\subsection{Augmentation products}

Let $\Delta=\ker\varepsilon_Q$ be the augmentation ideal of $k[Q]$.  Since
$(1-\varphi)G$ is $\varphi$-invariant,
\[
 \alpha(I)=I,
 \qquad
 \alpha\bigl(L_{(1-\varphi)G}\bigr)=L_{(1-\varphi)G},
 \qquad
 k[G]\beta(I)=L_{(1-\varphi)G}.
\]

\begin{theorem}[Augmentation transport]
\label{thm:alexander-powers}
For every $n\geq1$,
\[
 \Phi(\Delta^n)=I L_{(1-\varphi)G}^{\,n-1},
\]
where $\Delta^1=\Delta$ and $\Delta^{n+1}=\Delta^n\Delta$.  The
$k$-submodule generated by all $n$-fold products of elements of $\Delta$,
with arbitrary parenthesization, has the same image
$I L_{(1-\varphi)G}^{\,n-1}$.
\end{theorem}

\begin{proof}
Since $\Phi$ preserves coefficient sums, $\Phi(\Delta)=I$.  Suppose
\[
 \Phi(\Delta^n)=I L_{(1-\varphi)G}^{\,n-1}.
\]
The right-hand side is an ideal of $k[G]$ fixed by $\alpha$.  Hence
Proposition~\ref{prop:transported-product} gives
\[
 \Phi(\Delta^{n+1})
 =\alpha\bigl(\Phi(\Delta^n)\bigr)\beta(I)
 =I L_{(1-\varphi)G}^{\,n-1}k[G]\beta(I)
 =I L_{(1-\varphi)G}^{\,n}.
\]
This proves the left-normed formula.

For arbitrary parenthesizations, note that
$\beta(I)\subseteq L_{(1-\varphi)G}$ and
\[
 \beta\bigl(L_{(1-\varphi)G}\bigr)\subseteq L_{(1-\varphi)G},
\]
because $(1-\varphi)^2G\subseteq(1-\varphi)G$.  Suppose a parenthesized
product is obtained by multiplying subproducts with $i$ and $j$ factors.
If their images lie in $I L_{(1-\varphi)G}^{\,i-1}$ and
$I L_{(1-\varphi)G}^{\,j-1}$, respectively, then
\[
 \alpha\bigl(I L_{(1-\varphi)G}^{\,i-1}\bigr)
 \beta\bigl(I L_{(1-\varphi)G}^{\,j-1}\bigr)
 \subseteq I L_{(1-\varphi)G}^{\,i+j-1}.
\]
Induction on the product tree gives this containment for every
parenthesization.  The left-normed products are among these products and
already span $I L_{(1-\varphi)G}^{\,n-1}$, so equality follows.
\end{proof}

\subsection{Commutators}

Following~\cite[Section~5]{ST26}, let $k[Q]'$ denote the $k$-subalgebra
generated by all commutators $[u,v]=uv-vu$; no identity is adjoined.

\begin{theorem}[Commutator subalgebra]
\label{thm:commutator-subalgebra}
Under $\Phi:k[Q]\to k[G]$,
\[
 \Phi(k[Q]')=L_{(1-2\varphi)G}.
\]
\end{theorem}

\begin{proof}
Set $M=(1-2\varphi)G$.  Writing $y=x+z$ and applying
Proposition~\ref{prop:transported-product}, we obtain
\[
 \Phi([e_x,e_{x+z}])
 =X^{x+\varphi z}\bigl(X^{(1-2\varphi)z}-1\bigr).
\]
For fixed $z$, the exponent $x+\varphi z$ ranges over $G$.  Hence the
$k$-linear span of the basis commutators is
\[
 \left\langle
 X^g\bigl(X^{(1-2\varphi)z}-1\bigr):g,z\in G
 \right\rangle_k
 =L_M.
\]
By bilinearity, this is the span of all commutators.

Since $\varphi(M)=M$ and $(1-\varphi)M\subseteq M$, one has
$\alpha(L_M)=L_M$ and $\beta(L_M)\subseteq L_M$.  As $L_M$ is an ideal of
the ordinary group ring, it is closed under $\star$.  Therefore the span of
the commutators is already a subalgebra, and it equals $\Phi(k[Q]')$.
\end{proof}

\section{Augmentation quotients over cyclic groups}
\label{sec:quotients}

The augmentation formula reduces the case of a finite cyclic underlying
group to quotients of ideals in integral group rings.  For the core quandle
of the infinite cyclic group, the same method gives principal ideals in a
Laurent polynomial ring.

Let $m\geq2$ and let $a\in\mathbb{Z}$ satisfy $\gcd(a,m)=1$.
Write $\mathbb{Z}_m=\mathbb{Z}/m\mathbb{Z}$ and set
\[
 Q_{m,a}=\operatorname{Alex}(\mathbb{Z}_m,a),
 \qquad
 d=\gcd(m,1-a),
 \qquad
 q=m/d.
\]
Write $\Delta_{m,a}$ for the augmentation ideal of $\mathbb{Z}[Q_{m,a}]$.

\begin{theorem}[Augmentation quotients over cyclic groups]
\label{thm:cyclic-quotients}
There are isomorphisms of abelian groups
\[
 \frac{\Delta_{m,a}}{\Delta_{m,a}^2}
 \cong \mathbb{Z}^{d-1}\oplus\mathbb{Z}/q\mathbb{Z}
\]
and, for every $r\geq2$,
\[
 \frac{\Delta_{m,a}^r}{\Delta_{m,a}^{r+1}}
 \cong (\mathbb{Z}/q\mathbb{Z})^d.
\]
\end{theorem}

\begin{proof}
Put $C=\mathbb{Z}_m$ and $H=(1-a)C=dC\cong\mathbb{Z}_q$, and let $I$ and
$\mathfrak a$ be the augmentation ideals of $\mathbb{Z}[C]$ and
$\mathbb{Z}[H]$, respectively.  We use
$\mathfrak a^0=\mathbb{Z}[H]$.  Since
\[
L_H=\mathbb{Z}[C]\mathfrak a
\qquad\text{and hence}\qquad
L_H^{\,r-1}=\mathbb{Z}[C]\mathfrak a^{\,r-1},
\]
Theorem~\ref{thm:alexander-powers} gives
\[
\Phi(\Delta_{m,a}^r)=I\mathfrak a^{\,r-1}.
\]

Apply Hartl's decomposition
\cite[Theorem~1.5]{Hartl10} with $G=C$, subgroup $H$, coefficient ring
$\mathbb{Z}$, and left ideal
$J=\mathfrak a^{\,r-1}\subseteq\mathbb{Z}[H]$.  The theorem applies to
arbitrary subgroups and left ideals over a commutative coefficient ring,
and gives a noncanonical isomorphism
\[
\frac{I\mathfrak a^{\,r-1}}{I\mathfrak a^r}
\cong
\frac{\mathfrak a^r}{\mathfrak a^{r+1}}
\oplus
\left(\frac{I}{\mathbb{Z}[C]\mathfrak a}\right)
\otimes_{\mathbb{Z}}
\left(\frac{\mathfrak a^{\,r-1}}{\mathfrak a^r}\right).
\]

The quotient map $C\to C/H\cong\mathbb{Z}_d$ induces
\[
\frac{I}{\mathbb{Z}[C]\mathfrak a}
\cong
\ker\bigl(\mathbb{Z}[\mathbb{Z}_d]\to\mathbb{Z}\bigr),
\]
which is free abelian of rank $d-1$. Also,
\[
\mathbb{Z}[H]/\mathfrak a\cong\mathbb{Z}.
\]

If $q=1$, then $H$ is trivial and $\mathfrak a=0$, so
\[
\frac{\mathfrak a^s}{\mathfrak a^{s+1}}
=0
=\mathbb{Z}/q\mathbb{Z}
\qquad(s\geq1).
\]
If $q>1$, specializing \cite[Theorem~3.20]{CT11} to cyclic
$p$-groups and then using the Sylow decomposition \cite[(4.25)]{CT11}
gives
\[
\frac{\mathfrak a^s}{\mathfrak a^{s+1}}
\cong
\mathbb{Z}/q\mathbb{Z}
\qquad(s\geq1).
\]
Thus this formula holds for every $q\geq1$.

For $r=1$, Hartl's decomposition therefore gives
\[
\frac{\Delta_{m,a}}{\Delta_{m,a}^2}
\cong
\mathbb{Z}^{d-1}\oplus\mathbb{Z}/q\mathbb{Z}.
\]
For $r\geq2$, it gives
\[
\frac{\Delta_{m,a}^r}{\Delta_{m,a}^{r+1}}
\cong
\mathbb{Z}/q\mathbb{Z}
\oplus
\left(\mathbb{Z}^{d-1}\otimes_{\mathbb{Z}}\mathbb{Z}/q\mathbb{Z}\right)
\cong
(\mathbb{Z}/q\mathbb{Z})^d.
\]
\end{proof}

If $a\equiv1\pmod m$, then $d=m$, $q=1$, and
$L_{(1-a)\mathbb{Z}_m}=0$.  Thus
$\Delta_{m,a}\cong\mathbb{Z}^{m-1}$ and $\Delta_{m,a}^2=0$, in agreement
with the trivial-quandle multiplication.

\subsection{Dihedral quandles}

For an abelian group $G$, the core quandle
$\operatorname{Core}(G)=\operatorname{Alex}(G,-1)$ has operation
$x*y=2y-x$.  In particular,
$\operatorname{Dih}_m=\operatorname{Core}(\mathbb{Z}_m)$ is the dihedral
quandle of order $m$.

\begin{corollary}
\label{cor:dihedral-quotients}
Let $m\geq2$ and let $\Delta$ be the augmentation ideal of
$\mathbb{Z}[\operatorname{Dih}_m]$.  If $m$ is odd, then
\[
 \frac{\Delta^r}{\Delta^{r+1}}\cong\mathbb{Z}/m\mathbb{Z}
 \qquad(r\geq1).
\]
If $m$ is even, then
\[
 \frac{\Delta}{\Delta^2}
 \cong\mathbb{Z}\oplus\mathbb{Z}/(m/2)\mathbb{Z}
\]
and
\[
 \frac{\Delta^r}{\Delta^{r+1}}
 \cong\bigl(\mathbb{Z}/(m/2)\mathbb{Z}\bigr)^2
 \qquad(r\geq2).
\]
\end{corollary}

\begin{proof}
Apply Theorem~\ref{thm:cyclic-quotients} with $a=-1$, so
$d=\gcd(m,2)$ and $q=m/d$.
\end{proof}

The odd-order quotients and the even-order first quotient recover
\cite[Theorems~6.2 and~6.3]{EFT19}.  For $m=8$, the higher even-order
formula gives
$|\Delta^2/\Delta^3|=16$, in agreement with
\cite[Theorem~1]{PP24}.  More generally, putting $m=2n$ with $n\geq2$
proves the replacement conjecture~\cite[Conjecture~2]{PP24}:
every quotient in degree $r\geq2$ is
$(\mathbb{Z}/n\mathbb{Z})^2$.  Its order is $m^2/4$, replacing the
already-disproved order-$m$ prediction
in~\cite[Conjecture~6.5(2)]{BPS19}.  This comparison uses left-normed
powers, as in the counterexample computation; by
Theorem~\ref{thm:alexander-powers}, the spans over all parenthesizations
give the same quotients.

\subsection{The infinite core quandle}

\begin{corollary}
\label{cor:infinite-core}
Let $\Delta$ be the augmentation ideal of
$\mathbb{Z}[\operatorname{Core}(\mathbb{Z})]$.  Identify the underlying
$\mathbb{Z}$-module of this quandle ring with $\mathbb{Z}[t,t^{-1}]$
via $\Phi(e_j)=t^j$.  Then
\[
 \Phi(\Delta^n)
 =(t-1)(t^2-1)^{n-1}\mathbb{Z}[t,t^{-1}]
 \qquad(n\geq1)
\]
and
\[
 \Phi\bigl(\mathbb{Z}[\operatorname{Core}(\mathbb{Z})]'\bigr)
 =(t^3-1)\mathbb{Z}[t,t^{-1}].
\]
Moreover, as abelian groups,
\[
 \Delta/\mathbb{Z}[\operatorname{Core}(\mathbb{Z})]'\cong\mathbb{Z}^2,
\]
and $e_1-e_0\notin\Delta^2$.
\end{corollary}

\begin{proof}
The two ideal descriptions are the specializations of
Theorems~\ref{thm:alexander-powers} and~\ref{thm:commutator-subalgebra} to
$\varphi=-1$.  Since $t^3-1=(t-1)(t^2+t+1)$, one has
$\mathbb{Z}[\operatorname{Core}(\mathbb{Z})]'\subseteq\Delta$.
Multiplication by $t-1$ identifies
$\mathbb{Z}[t,t^{-1}]$ with $(t-1)\mathbb{Z}[t,t^{-1}]$, and under this
identification
$\Phi\bigl(\mathbb{Z}[\operatorname{Core}(\mathbb{Z})]'\bigr)$ corresponds
to $(t^2+t+1)$.  Hence
\[
 \Delta/\mathbb{Z}[\operatorname{Core}(\mathbb{Z})]'
 \cong\mathbb{Z}[t,t^{-1}]/(t^2+t+1)
 \cong\mathbb{Z}^2,
\]
the last group having basis represented by $1$ and $t$.  In particular,
the commutator subalgebra is strictly smaller than $\Delta$.

Finally, $\Phi(e_1-e_0)=t-1$.  If $e_1-e_0\in\Delta^2$, then
$t-1=(t-1)(t^2-1)f$ for some $f\in\mathbb{Z}[t,t^{-1}]$.  Cancellation in
the Laurent polynomial domain would give $1=(t^2-1)f$, which is impossible
after evaluating at $t=1$.
\end{proof}

The formula at $n=2$ describes the square asked for
in~\cite[Question~3.4]{BE26}, and $e_1-e_0\notin\Delta^2$ gives a negative
answer to its membership subquestion.  The strict commutator inclusion
separately gives a negative answer to the reverse-inclusion question
in~\cite[Section~3.1]{BE26}.

\section{Separated integral filtrations and idempotents}
\label{sec:separation}

Separation of a subgroup ideal forces separation of the corresponding
augmentation filtration.  This rules out nonzero idempotents of
augmentation zero.  We include zero among idempotents unless explicitly
excluded.

\begin{theorem}[$p$-primary separation]
\label{thm:p-primary-separation}
Let $G$ be an abelian group, let $\varphi\in\operatorname{Aut}(G)$, and
suppose that $(1-\varphi)G$ is a finite $p$-group. If $\Delta$ is the
augmentation ideal of $\mathbb{Z}[\operatorname{Alex}(G,\varphi)]$, then
\[
 \bigcap_{n\geq1}\Delta^n=0.
\]
Consequently, every nonzero idempotent of
$\mathbb{Z}[\operatorname{Alex}(G,\varphi)]$ has augmentation one.
\end{theorem}

\begin{proof}
Put $H=(1-\varphi)G$ and let $\mathfrak a_H$ be the augmentation ideal of
$\mathbb{Z}[H]$.  We first recall the finite-$p$-group argument
in~\cite[proof of Lemma~3.2]{Emmanouil25}, with coefficient ring
$\mathbb{Z}$.  The augmentation ideal of $\mathbb{F}_p[H]$ has finitely
many commuting generators $X^h-1$, each satisfying
$(X^h-1)^{|H|}=0$, so it is nilpotent.  Hence for some $N\geq1$,
$\mathfrak a_H^N\subseteq p\mathbb{Z}[H]$, and therefore
$\mathfrak a_H^{Nr}\subseteq p^r\mathbb{Z}[H]$ for every $r\geq1$.
Since $\mathbb{Z}[H]$ is free abelian,
\[
 \bigcap_{n\geq1}\mathfrak a_H^n
 \subseteq\bigcap_{r\geq1}p^r\mathbb{Z}[H]=0.
\]

Choose a set $\mathcal T$ of representatives for the cosets of $H$ in
$G$.  Then
\[
 \mathbb{Z}[G]=\bigoplus_{c\in\mathcal T}X^c\mathbb{Z}[H],
 \qquad
 L_H^n=\bigoplus_{c\in\mathcal T}X^c\mathfrak a_H^n.
\]
An element lying in every $L_H^n$ has each of its finitely many coset
components in every $\mathfrak a_H^n$.  Thus
$\bigcap_{n\geq1}L_H^n=0$.  By
Theorem~\ref{thm:alexander-powers},
$\Phi(\Delta^n)=I L_H^{n-1}\subseteq L_H^{n-1}$ for $n\geq2$, so the
augmentation filtration is separated.

Let $u$ be an idempotent. Since augmentation is an algebra homomorphism,
$\varepsilon_Q(u)\in\{0,1\}$. If $\varepsilon_Q(u)=0$, then $u\in\Delta$;
from $u^2=u$ it follows inductively that $u\in\Delta^n$ for every $n$.
Separation then forces $u=0$.
\end{proof}

The idempotent conclusion of Theorem~\ref{thm:p-primary-separation}
establishes Conjecture~6.2 of \cite{ENS23} for every Alexander quandle
$\operatorname{Alex}(G,\varphi)$ such that $(1-\varphi)G$ is a finite
$p$-group.

For a finite quandle with more than two elements, cyclic type means that
each right translation $R_y$ restricts to a single cycle on
$Q\setminus\{y\}$.  Quandles of orders at most two are trivial and will be
treated separately.

\begin{corollary}[Quandles of cyclic type]
\label{cor:cyclic-type}
Every nonzero idempotent in the integral quandle ring of a finite quandle of
cyclic type has augmentation one.
\end{corollary}

\begin{proof}
For a trivial quandle, $u^2=\varepsilon_Q(u)u$, so every nonzero integral
idempotent has augmentation one.  This covers orders at most two.
Suppose $|Q|>2$.  Vendramin's classification
\cite[Corollary~2]{V17} identifies a quandle of cyclic type with an
Alexander quandle over a finite field.  His left-quandle convention uses
$x\triangleright y=(1-\lambda)x+\lambda y$, where $\lambda$ has order
$|Q|-1$ by~\cite[Proposition~11]{V17}.  To pass to our right-quandle
convention, reverse the arguments:
\[
 x*y=y\triangleright x=\lambda x+(1-\lambda)y.
\]
Thus $\varphi$ is multiplication by $\lambda$.
Since $\lambda\ne1$ for $|Q|>2$,
$(1-\varphi)G=(1-\lambda)G=G$, whose additive group is a finite abelian
$p$-group.  Theorem~\ref{thm:p-primary-separation} applies.
\end{proof}

Thus Conjecture~6.2 of \cite{ENS23} holds, in particular, for every finite
quandle of cyclic type.

\section{Fourier methods}
\label{sec:fourier}

Fourier coordinates turn the idempotent equations into relations among
characters.  For midpoint quandles these relations control integral
coefficients, while for core quandles they determine the spectra of right
multiplication operators.

\subsection{Fourier transform of the transported product}

Let $G$ be a finite abelian group, let
$\varphi\in\operatorname{Aut}(G)$, and let
$\widehat G=\operatorname{Hom}(G,\mathbb{C}^\times)$ be its character group.
For $u=\sum_{g\in G}a_g X^g\in\mathbb{C}[G]$, define
\[
 \widehat u(\chi)=\sum_{g\in G}a_g\chi(g)
 \qquad(\chi\in\widehat G).
\]

\begin{proposition}[Fourier form of the Alexander product]
\label{prop:fourier-formulas}
For $u,v\in\mathbb{C}[G]$,
\[
 \widehat{u\star v}(\chi)
 =\widehat u(\chi\circ\varphi)
  \widehat v(\chi\circ(1-\varphi))
 \qquad(\chi\in\widehat G).
\]
\end{proposition}

\begin{proof}
For the ordinary group-ring product,
$\widehat{uv}(\chi)=\widehat u(\chi)\widehat v(\chi)$.  If
$\psi\in\operatorname{End}(G)$ and $\psi_*(X^g)=X^{\psi(g)}$, then
$\widehat{\psi_*(u)}(\chi)=\widehat u(\chi\circ\psi)$.  Apply these
identities to $u\star v=\alpha(u)\beta(v)$.
\end{proof}

We shall also use the unnormalized Parseval identity
\[
 \sum_{\chi\in\widehat G}|\widehat u(\chi)|^2
 =|G|\sum_{g\in G}|a_g|^2,
\]
which follows immediately from character orthogonality.

\subsection{Commutative and medial quandles}

A quandle is \emph{medial} if
\[
 (x*y)*(z*w)=(x*z)*(y*w)
\]
for all $x,y,z,w$.  Every Alexander quandle is medial.  Ta proved that
every medial commutative quandle is isomorphic to a midpoint quandle
\cite[Corollary~4.2]{Ta26}.  We first classify integral idempotents in the
finite midpoint case and then transfer the result to all finite medial
commutative quandles.

Assume that $G$ has odd order.  Let $[2]:G\to G$, $g\mapsto2g$, and set
$\tau=[2]^{-1}$.  Since $1-\tau=\tau$, the quandle
$M(G)=\operatorname{Alex}(G,\tau)$ has operation $x*y=(x+y)/2$ and is
commutative.  Conversely, commutativity of an Alexander operation forces
$2\varphi=1$ by setting $y=0$, so this is exactly the commutative
Alexander quandle on $G$; see also~\cite[Section~2]{ST26}.  Write
$\operatorname{Idem}(\mathcal A)=\{u\in\mathcal A:u^2=u\}$ for the
idempotent set of an algebra, including zero.

\begin{theorem}[Integral idempotents of commutative Alexander quandles]
\label{thm:midpoint-idempotents}
Let $G$ be a finite abelian group of odd order.  Then
\[
 \operatorname{Idem}(\mathbb{Z}[M(G)])
 =\{0\}\cup\{e_g:g\in G\}.
\]
\end{theorem}

\begin{proof}
Let $u=\sum_{g\in G}a_g e_g\in\mathbb{Z}[M(G)]$ be idempotent, identify it
with its image in $\mathbb{Z}[G]\subseteq\mathbb{C}[G]$, and put
$c_\chi=\widehat u(\chi)$.  Define a permutation $\sigma$ of $\widehat G$
by $\sigma(\chi)=\chi\circ\tau$.  Since $\varphi=1-\varphi=\tau$,
Proposition~\ref{prop:fourier-formulas} shows that idempotency is equivalent
to
\[
 c_\chi=c_{\sigma(\chi)}^{\,2}
 \qquad(\chi\in\widehat G).
\]
If the $\sigma$-orbit of $\chi$ has length $r$, iteration gives
$c_\chi=c_\chi^{2^r}$.  Thus $c_\chi=0$ or $|c_\chi|=1$.

Parseval's identity now yields
\[
 |G|\sum_{g\in G}a_g^2
 =\sum_{\chi\in\widehat G}|c_\chi|^2
 \leq|G|.
\]
Hence $\sum_g a_g^2\leq1$.  If $u\ne0$, integrality forces
$u=e_g$ or $u=-e_g$ for some $g$.  Every $e_g$ is idempotent, whereas
$(-e_g)\star(-e_g)=e_g\ne-e_g$.
\end{proof}

\begin{corollary}[Finite medial commutative quandles]
\label{cor:medial-commutative-idempotents}
Let $Q$ be a finite medial commutative quandle.  Then
\[
 \operatorname{Idem}(\mathbb{Z}[Q])
 =\{0\}\cup\{e_x:x\in Q\}.
\]
\end{corollary}

\begin{proof}
By~\cite[Corollary~4.2]{Ta26}, there are an abelian group $G$ on which
multiplication by $2$ is an automorphism and a quandle isomorphism
$Q\cong M(G)$.  Since $Q$ is finite, so is $G$.  A finite abelian group on
which multiplication by $2$ is injective has odd order, so
Theorem~\ref{thm:midpoint-idempotents} applies.  Finally, a quandle
isomorphism extends $\mathbb{Z}$-linearly to an isomorphism of quandle
rings and carries basis elements to basis elements.
\end{proof}

\begin{corollary}[Commutative Alexander quandles on cyclic groups]
\label{cor:cyclic-midpoint}
For every odd integer $m\geq3$,
\[
 \operatorname{Idem}(\mathbb{Z}[M(\mathbb{Z}_m)])
 =\{0,e_0,e_1,\ldots,e_{m-1}\}.
\]
\end{corollary}

\begin{proof}
Apply Theorem~\ref{thm:midpoint-idempotents} with $G=\mathbb{Z}_m$.
\end{proof}

If $m=2n+1$, then $[2]^{-1}$ is multiplication by $n+1$ modulo $m$, so this
is the commutative quandle considered in~\cite[Example~5.1(1)]{BE26}.
Corollary~\ref{cor:cyclic-midpoint} therefore answers the integral
instance of~\cite[Question~5.5]{BE26}: the only nonzero integral
idempotents are the basis elements.  It also extends the question's
prime-order range to every odd order; no classification over arbitrary
coefficient domains is asserted.

The augmentation-one equations for $M(\mathbb{Z}_5)$ were studied by
Gr\"obner-basis methods in~\cite[Section~4, pp.~15--17]{CCFK26}.\footnote{In
the cited version, the listed vector $(-1,0,0,0)$ gives
$q_1(-1,0,0,0)=-2$ for the first polynomial on p.~15, so the printed
solution list has a sign inconsistency.  The Fourier proof above is
independent of that calculation.}
Theorem~\ref{thm:cyclic-quotients} also shows that, if $\Delta$ denotes the
augmentation ideal of $\mathbb{Z}[M(\mathbb{Z}_m)]$, then
\[
 \Delta^r/\Delta^{r+1}\cong\mathbb{Z}/m\mathbb{Z}
 \qquad(r\geq1),
\]
because $1-[2]^{-1}=[2]^{-1}$ is an automorphism of $\mathbb{Z}_m$.

\subsection{Core quandles and Peirce spectra}

Let $\mathcal A$ be a finite-dimensional complex nonassociative algebra.
For a nonzero idempotent $u\in\mathcal A$, let
$\rho_u:\mathcal A\to\mathcal A$, $v\mapsto vu$, be right multiplication
by $u$, and define
\[
 \operatorname{Spec}_{\mathrm P}(\mathcal A)
 =\bigcup_{\substack{u\in\mathcal A\\0\ne u^2=u}}
   \operatorname{Spec}(\rho_u).
\]
This follows the convention of~\cite[Section~7]{ENS23}: right
multiplication by nonzero idempotents.

Let $G$ be a finite abelian group of odd order and put
$Q=\operatorname{Core}(G)$.  For $\chi\in\widehat G$, write
$\chi^{-1}(g)=\chi(g)^{-1}$ and $\chi^2(g)=\chi(2g)$.

\begin{theorem}[Peirce spectrum of an odd-order core]
\label{thm:core-peirce}
Let $G\ne\{0\}$ be a finite abelian group of odd order.  Then
\[
 \operatorname{Spec}_{\mathrm P}
 \bigl(\mathbb{C}[\operatorname{Core}(G)]\bigr)=\{0,1,-1\}.
\]
\end{theorem}

\begin{proof}
Let $0\ne u\in\mathbb{C}[\operatorname{Core}(G)]$ be idempotent, identify
it with its image in $\mathbb{C}[G]$, and put $c_\chi=\widehat u(\chi)$.
Since $\varphi=-1$ and
$1-\varphi=2$, Proposition~\ref{prop:fourier-formulas} gives
\[
 \widehat{v\star w}(\chi)
 =\widehat v(\chi^{-1})\widehat w(\chi^2).
\]
Thus $u\star u=u$ is equivalent to
\begin{equation}
\label{eq:core-idempotent-fourier}
 c_\chi=c_{\chi^{-1}}c_{\chi^2}
 \qquad(\chi\in\widehat G).
\end{equation}

Let $S=\{\chi:c_\chi\ne0\}$.  Equation
\eqref{eq:core-idempotent-fourier} shows that $S$ is closed under inversion
and squaring.  Since squaring is a permutation of the finite group
$\widehat G$, its restriction to $S$ is a permutation, so $S$ is also
closed under inverse squaring.  Put $P_\chi=c_\chi c_{\chi^{-1}}$.
Multiplying \eqref{eq:core-idempotent-fourier} for $\chi$ and $\chi^{-1}$
gives $P_\chi=P_\chi P_{\chi^2}$.  If $\chi\in S$, then $P_\chi\ne0$ and
hence $P_{\chi^2}=1$.  Since squaring permutes $S$,
\[
 P_\eta=1
 \qquad(\eta\in S).
\]

For right multiplication by $u$,
\[
 \widehat{\rho_u(v)}(\chi)
 =c_{\chi^2}\widehat v(\chi^{-1}).
\]
Because $G$ has odd order, the inversion orbits in $\widehat G$ consist of
the trivial character and pairs $\{\chi,\chi^{-1}\}$ with $\chi\ne1$.
On the coordinates indexed by such a pair, the matrix of $\rho_u$ is
\[
 \begin{pmatrix}
  0 & c_{\chi^2}\\
  c_{\chi^{-2}} & 0
 \end{pmatrix}.
\]
If $\chi\in S$, then
$c_{\chi^2}c_{\chi^{-2}}=P_{\chi^2}=1$, so this block has eigenvalues
$1$ and $-1$.  If $\chi\notin S$, closure under inverse squaring implies
$\chi^2,\chi^{-2}\notin S$, so the block is zero.  On the
trivial-character coordinate, the eigenvalue is $c_1$, and
\eqref{eq:core-idempotent-fourier} gives $c_1=c_1^2$.  Every Peirce
eigenvalue therefore belongs to $\{0,1,-1\}$.

All three values occur.  The basis element $e_0$ is idempotent, and right
multiplication by $e_0$ sends $e_g$ to $e_{-g}$.  Since $G$ is nontrivial
and has odd order, inversion has a two-element orbit, whose permutation
matrix has eigenvalues $1$ and $-1$.  Finally, set
\[
 \omega=\frac1{|G|}\sum_{g\in G}e_g.
\]
Right translations permute $G$, so $\omega^2=\omega$.  Since multiplication
by $2$ is bijective on $G$, one also has
$\rho_\omega(v)=\varepsilon_Q(v)\omega$.  The nonzero augmentation-zero
subspace is contained in $\ker\rho_\omega$, so $0$ occurs.
\end{proof}

\begin{corollary}[Odd-order dihedral quandles]
\label{cor:dihedral-peirce}
For every odd integer $m\geq3$,
\[
 \operatorname{Spec}_{\mathrm P}
 \bigl(\mathbb{C}[\operatorname{Dih}_m]\bigr)=\{0,1,-1\}.
\]
\end{corollary}

\begin{proof}
Since $\operatorname{Dih}_m=\operatorname{Core}(\mathbb{Z}_m)$, apply
Theorem~\ref{thm:core-peirce} with $G=\mathbb{Z}_m$.
\end{proof}

This proves~\cite[Conjecture~7.6]{ENS23}, with its convention of right
multiplication by nonzero idempotents.

\section*{Declaration of AI use}

OpenAI's ChatGPT was used during the development of this work for assistance
with literature searches, exploration and verification of mathematical arguments,
and manuscript preparation.  All mathematical claims, proofs, citations, and the
final text were independently checked and approved by the author, who takes full
responsibility for the work.

\end{document}